\documentclass[12pt]{article}
\usepackage{enumerate}
\usepackage{amsthm,amsmath,amssymb,latexsym}
\usepackage[T2A]{fontenc}
\usepackage[cp1251]{inputenc}
\usepackage[english]{babel}
\usepackage{tikz}
\usetikzlibrary{calc,intersections,through}

\newtheorem{Lemma}{Lemma}[section]
\newtheorem{theorem}[Lemma]{Theorem}
\newtheorem{corollary}[Lemma]{Corollary}

\theoremstyle{definition}
\newtheorem{definition}[Lemma]{Definition}

\theoremstyle{remark}
\newtheorem{example}[Lemma]{Example}
\newtheorem{remark}[Lemma]{Remark}
\newtheorem{problem}[Lemma]{Problem}

\allowdisplaybreaks[4]
\numberwithin{equation}{section}

\begin{document}
\title{Density by moduli and Wijsman statistical convergence}

\author{Vinod K. Bhardwaj, Shweta Dhawan \and and Oleksiy A. Dovgoshey}

\date{}
\maketitle
\begin{abstract}
In this paper, we generalized the Wijsman statistical convergence of closed sets in metric space by introducing the $f$-Wijsman statistical convergence these of sets, where $f$ is an unbounded modulus. It is shown that the Wijsman convergent sequences are precisely those sequences which are $f$-Wijsman statistically convergent for every unbounded modulus $f$. We also introduced a new concept of Wijsman strong Ces\`{a}ro summability with respect to a modulus, and investigate the relationships between the $f$-Wijsman statistically convergent sequences and the Wijsman strongly Ces\`{a}ro summable sequences with respect to $f$.
\end{abstract}

\medskip


\section{Introduction and background}
The idea of statistical convergence was first introduced by Fast \cite{hf51} and Steinhaus \cite{hs51} independently in the same year 1951 and since then several generalizations and applications of this concept have been investigated by various authors, namely $\check{S}a$l$\Acute{a}$t \cite{ts80}, Fridy \cite{jf85}, Connor \cite{jc88}, Aizpuru $et~al.$ \cite{ab14}, K\"{u}\c{c}\"{u}kaslan $et~al.$ \cite{md14}, and many others.

Statistical convergence depends on the natural density of subsets of the set $\mathbb{N} = \{1,2,3,\ldots\}$. The natural density $d(K)$ of set $K \subseteq \mathbb{N}$ (see \cite[Chapter~11]{iz80}) is defined by
\begin{equation}\label{eq1.1}
d(K) = \lim_{n \to \infty}\frac{1}{n}\left|\{k\leq n\colon k \in K\}\right|,
\end{equation}
where $\left|\{\,k\leq n: k \in K \}\right|$ denotes the number of elements of $K$ not exceeding~$n$. Obviously we have $d(K) =0$ provided that $K$ is finite.

In what follows we write $(x_k) \subset A$ if all elements of the sequence $(x_k)$ belong to $A$.

\begin{definition}
A sequence $(x_{k}) \subset \mathbb R$ is said to be statistically convergent to $l \in \mathbb R$ if, for each $\varepsilon>0$, the set $\{k \in \mathbb{N}\colon |x_k - l| \geq \varepsilon\}$ has the zero natural density.
\end{definition}

A new concept of density by moduli was introduced by Aizpuru $et~al.$\cite{ab14} that enabled them to obtain a nonmatrix method of convergence, namely, the $f$-statistical convergence which is a generalization of statistical convergence.

We recall that a modulus is a function $f\colon [0, \infty) \to [ 0, \infty)$ such that
\begin{enumerate}
\item[$(i)$] $f(x) = 0$ if and only if $ x = 0$,
\item [$(ii)$] $f(x + y) \leq f(x) + f(y)$ for all $x, y \in [0,\infty)$,
\item [$(iii)$] $f$ is increasing,
\item [$(iv)$] $f$ is continuous.
\end{enumerate}

The functions $f$ satisfying condition $(ii)$ are called subadditive. If $f$, $g$ are moduli and $a$, $b$ are positive real numbers, then
$$
f\circ g,\quad  af+bg, \quad\text{and}\quad f\vee g
$$
are moduli. A modulus may be unbounded or bounded. For example, the modulus $f(x) = x^p$ where $0 < p \leq 1$, is unbounded, but $g(x) = \frac{x}{(1+ x)}$ is bounded. It is interesting to note that $f\colon [0, \infty) \to [ 0, \infty)$ is a modulus if and only if there is an uniformly continuous, non-constant function $g\colon [0,\infty) \to [0,\infty)$ such that 
$$
f(t) = \sup_{\substack{|x-y|\leq t \\ x, y \in [0,\infty)}} |g(x)-g(y)|
$$
holds for every $t \in [0, \infty)$. The details can be found in Dovgoshey \emph{et al.} \cite[Theorem~4.3]{DM}. For bounded moduli this characterization has been, in fact, known Lebesgue~\cite{Le} in~$1910$.

The idea of replacing of natural density with density by moduli, has motivated us to look for some new generalizations of statistical convergence \cite{vd15, vd15a}. Using the density by moduli Bhardwaj $et~ al.$ \cite{vg15} have also introduced the concept of $f$- statistical boundedness which is a generalization of the concept of statistical boundedness \cite{jo97} and intermediate between the usual boundedness and the statistical boundedness.

The concept of convergence of sequences of points has been extended by several authors \cite{jf90, mp86, gb85, gb94, IM2015, yz93, yz94, rw64, rw66} to convergence of sequences of sets. One of such extensions considered in this paper is the concept of Wijsman convergence. Nuray and Rhoades \cite{fr12} extended the notion of Wijsman convergence of sequences of sets to that of Wijsman statistical convergence and introduced the notion of Wijsman strong Ces\`{a}ro summability of sequences of sets and discussed its relations with Wijsman statistical convergence.

In this paper we extend the Wijsman statistical convergence to a $f$-Wijsman statistical convergence, where $f$ is an unbounded modulus.

Let us recall the basic definitions of $f$-density and $f$-statistical convergence.

\begin{definition}[\cite{ab14}]\label{D:01}
Let $f\colon [0, \infty) \to [0,\infty)$ be an unbounded modulus. The $f$-density $d^f(K)$ of a set $K \subseteq \mathbb{N}$ is defined as
\begin{equation}\label{eq1.2}
d^{f} (K) := \lim_{n \to \infty} \frac{f(|\{k \leq n\colon k \in K\}|)}{f(n)}
\end{equation}
if this limit exists. A sequence $(x_k) \subset \mathbb R$ is said to be $f$-statistically convergent to $l \in \mathbb R$ if, for each $\varepsilon >0$, the set $\{k \in \mathbb N\colon |x_k-l|\geq \varepsilon\}$ has the zero $f$-density.
\end{definition}

\begin{remark}
For each unbounded modulus $f$, the finite sets have the zero $f$-density and 
$$
(d^{f}(K)=0) \Rightarrow  (d^{f}(\mathbb{N}-K) = 1)
$$
holds for every $K \subseteq \mathbb N$ but, in general, the implication
$$
(d^{f} (\mathbb{N} -K) = 1) \Rightarrow (d^{f} (K)=0)
$$
does not hold. For example if we take $f(x)=\log(1+x)$ and $K=\{2n\colon n\in \mathbb N\}$, then 
$$
d^f(K)=d^f(\mathbb N-K)=1.
$$
\end{remark}

\begin{example}\label{E:1.4}
A set having the zero natural density may have a non-zero $f$-density. In particular
$$
d(K) = 0 \quad \text{and}\quad d^{f}(K) = 1/2
$$
holds for $f(x)= \log{(1+x)}$ and $K = \{n^2\colon n \in \mathbb N\}$.
\end{example}

Now we pause to collect some definitions related to Wijsman convergence of sequences of sets in a metric space.

Let $(X,\rho)$ be a metric space with a metric $\rho$. For any $x \in X$ and any non-empty set $A \subseteq X$, the distance from $x$ to $A$ is defined by
\begin{align*}
d(x,A)= \inf_{y \in A}\rho(x,y).
\end{align*}
In what follows we denote by $CL(X)$ the set of all non-empty closed subsets of $(X, \rho)$.

\begin{definition}\label{D:07}
Let $(X,\rho)$ be a metric space, $(A_k) \subset CL(X)$ and $A \in CL(X)$. Then $(A_k)$ is said to be:
\begin{itemize}
	\item \emph{Wijsman convergent to $A$}, if the numerical sequence $(d(x,A_k))$ is convergent to $d(x,A)$ for each $x \in X$;
	\item \emph{Wijsman statistically convergent} to $A \in CL(X)$, if for each $x \in X$, the numerical sequence $(d(x,A_k))$ is statistically convergent to $d(x,A)$;
	\item \emph{Wijsman bounded} if
	\begin{equation}\label{eq1.3}
	\sup_k d(x,A_k) < \infty
	\end{equation}
	for each $x \in X$;
	\item \emph{Wijsman Ces\`{a}ro summable to $A$} if, for each $x \in X$, the sequence $(d(x,A_k))$ is Ces\`{a}ro summable to $d(x,A)$, i.e.,
	\begin{align*}
	\lim_{n \to \infty} \frac{1}{n}\sum_{k=1}^{n}d(x,A_k)=d(x,A);
	\end{align*}
	\item \emph{Wijsman strongly Ces\`{a}ro summable to $A$} if, for each $x \in X$, the sequence~$(d(x,A_k))$ is strongly Ces\`{a}ro summable to $d(x,A)$, i.e.,
	\begin{align*}
	\lim_{n \to \infty}\frac{1}{n}\sum_{k=1}^{n}|d(x,A_k)-d(x,A)|=0.
	\end{align*}
\end{itemize}
\end{definition}

\begin{remark}\label{R1.6}
	The sets $A_k$ belonging to a Wijsman bounded sequence $(A_k)$ can be unbounded subsets of $(X, \rho)$, i.e.,
	$$
	\operatorname{diam} A_k =\sup\{\rho(x,y)\colon x, y \in A_k\}=\infty.
	$$
	Moreover, the triangle inequality implies that $(A_k)$ is Wijsman bounded if there exists at least one point $p \in X$ such that~\eqref{eq1.3} holds with $x=p$.
\end{remark}

\begin{example}\label{R:03}
	Let $(X, \rho)$ be the complex plane $\mathbb C$ with the standard metric. Let us consider the sequence $(A_k)$ defined as follows:
	\[
	A_{k} :=
	\begin{cases}
	\left\{z \in \mathbb C\colon |z-1|=\frac{1}{k}\right\},& \text{if $k$ is a square,}\\
	\{0\}, &\text{otherwise}.
	\end{cases}
	\]
	This sequence is Wijsman statistically convergent to $\{0\}$ but not Wijsman convergent.
\end{example}

\begin{definition}\label{D:1}
	Let $(X,\rho)$ be a metric space, let $(A_k) \subset CL(X)$ and let $f\colon [0,\infty) \to [0,\infty)$ be an unbounded modulus. The sequence $ (A_{k})$ is said to be $f$-Wijsman statistically convergent to $A \in CL(X)$ if the sequence $(d(x,A_k))$ is $f$-statistically convergent to $d(x,A)$ for each $x \in X$. 
\end{definition}

We write 
$$
[WS^{f}]-\lim A_k = A
$$
if $(A_k)$ is $f$-Wijsman statistically convergent to $A$. In the case where $f(x) = ax$, $a >0$, the $f$-Wijsman statistical convergence reduces to the Wijsman statistical convergence. 

We prove that the Wijsman convergent sequences are precisely those sequences which are $f$-Wijsman statistically convergent for every unbounded modulus $f$. We also introduce a new concept of Wijsman strong Ces\`{a}ro summability with respect to a modulus and show that if a sequence is Wijsman strongly Ces\`{a}ro summable, then it is Wijsman strongly Ces\`{a}ro summable with respect to all moduli $f$. The moduli $f$ for which the converse is true are investigated. Finally, we study a relation between Wijsman strong Ces\`{a}ro summability with respect to a modulus $f$ and $f$-Wijsman statistical convergence.

\section{$f$-Wijsman statistical convergence}

The results of this section are closely related with paper~\cite{ab14}.

\begin{theorem}\label{T:2}
	Let $f\colon [0,\infty) \to [0,\infty)$ be an unbounded modulus, $(X,\rho)$ be a metric space, $A \in CL(X)$ and let $(A_k) \subset CL(X)$ such that
	\begin{equation}\label{T:2e1}
	[WS^f]-\lim A_k = A.
	\end{equation}
	Then $(A_k)$ is Wijsman statistically convergent to $A$.
\end{theorem}
\begin{proof}
	For all $x \in X$, $\varepsilon>0$ and $n \in \mathbb N$ we write
	$$
	K_{x,\varepsilon}(n):=\{k\leq n\colon |d(x,A_k)-d(X,A)|\geq \varepsilon\}.
	$$
	If $(A_k)$ is not Wijsman statistically convergent to $A$, then there are $x \in X$ and $\varepsilon>0$ such that 
	$$
	\limsup_{n\to \infty} \frac{|K_{x,\varepsilon}(n)|}{n}>0.
	$$
	Hence there exist $p \in \mathbb N$ and a sequence $(n_m) \subset \mathbb N$, such that
	\begin{equation}\label{T:2e2}
	\lim_{m\to \infty} n_m = \infty
	\end{equation}
	and
	$$
	\frac{1}{n_m} |K_{x,\varepsilon}(n_m)| \geq \frac{1}{p}
	$$
	for every $m \in \mathbb N$. The last inequality is equivalent to
	\begin{equation}\label{T:2e3}
	n_m \leq p\, |K_{x,\varepsilon}(n_m)|.
	\end{equation}
	Using the subadditivity of $f$ and~\eqref{T:2e3} we obtain
	$$
	f(n_m) \leq p\, f(|K_{x,\varepsilon}(n_m)|).
	$$
	Consequently the inequality
	\begin{equation}\label{T:2e4}
	\frac{f(|K_{x,\varepsilon}(n_m)|)}{f(n_m)} \geq \frac{1}{p}
	\end{equation}
	holds for every $m \in \mathbb N$. Equality~\eqref{T:2e2} and inequality~\eqref{T:2e4} imply 
	$$
	\limsup_{n\to \infty} \frac{f(|K_{x,\varepsilon}(n)|)}{f(n)}\geq \frac{1}{p},
	$$
	contrary to~\eqref{T:2e1}.
\end{proof}

\begin{remark}\label{R:2.6}
	Using Example~\ref{E:1.4} it is easy to construct a Wijsman statistically convergent sequence which is not $f$-Wijsman statistically convergent with $f(x)=\log (1+x)$.
\end{remark}

\begin{theorem}\label{T:3}
	Let $(X,\rho)$ be a metric space and $f$, $g$ be unbounded moduli. Then for all $A$, $B \in CL(X)$ and every $(A_k) \subset CL(X)$ the equalities 
	\begin{equation}\label{e2.5}
	[WS^f]-\lim A_k = A \quad \text{and} \quad [WS^g]-\lim A_k = B
	\end{equation}
	imply $A=B$.
\end{theorem}
\begin{proof}
	Let $(X,\rho)$ be a metric space, let $(A_k) \subset CL(X)$ and let~\eqref{e2.5} hold. By Theorem~\ref{T:2} the sequence $(A_k)$ is Wijsman statistically convergent to $A$ and to $B$. Using the uniqueness of statistical limits of numerical sequences we obtain that $d(x,A)=d(x,B)$ holds for every $x \in X$. It implies the equality $A=B$ because $A$, $B \in CL(X)$.
\end{proof}

\begin{corollary}
	Let $(X,\rho)$ be a metric space and let $(A_k) \subset CL(X)$. Then for every unbounded modulus $f\colon [0,\infty)\to [0,\infty)$, the limit
	$$
	[WS^f]-\lim A_k
	$$
	is unique if it exists.
\end{corollary}

We will say that a modulus $f\colon [0,\infty)\to [0,\infty)$ is slowly varying if the limit relation 
\begin{equation}\label{eq2.6}
\lim_{x\to \infty} \frac{f(ax)}{f(x)} =1
\end{equation}
holds for every $a>0$. (See Seneta \cite[Chapter~1]{Sen} for the properties of slowly varying functions.) It is clear that all bounded modulus are slowly varying. The function $f(x)=\log(1+x)$ is an example of unbounded slowly varying modulus.

The following lemma is a refinement of Lemma~3.4 from~\cite{ab14}.

\begin{Lemma}\label{L:1}
	Let $K$ be an infinite subset of $\mathbb N$. Then there is an unbounded, concave and slowly varying modulus $f\colon [0,\infty)\to [0,\infty)$ such that
	\begin{equation}\label{L:1e1}
	d^f(K)=1.
	\end{equation}
\end{Lemma}
\begin{proof}
	For every $n \in \mathbb N$ write
	$$
	K(n):=\{m\in K\colon m\leq n\}.
	$$
	Since $K$ is infinite, there is a sequence $(n_k) \subset \mathbb N$ such that:
	\begin{equation}\label{L:1e2}
	\lim_{k\to \infty}\frac{n_{k+1}}{n_k} = \infty
	\end{equation}
	and 
	\begin{equation}\label{L:1e3}
	n_{k+1} - n_{k} < n_{k+2} - n_{k+1}, \quad 2n_{k} < n_{k+1}
	\end{equation}
	and 
	\begin{equation}\label{L:1e4}
	n_k < \left|K(n_{k+1})\right|
	\end{equation}
	hold for every $k \in \mathbb N$.
	
	Write $n_0=0$ and define a function $f\colon [0,\infty)\to [0,\infty)$ by the rule: if $x \in [n_{k-1},n_k]$, $k \in \mathbb N$, then
	\begin{equation}\label{L:1e5}
	f(x)=\frac{x-n_{k-1}}{n_k - n_{k-1}} + k-1.
	\end{equation}
	In particular, we have 
	\begin{equation}\label{L:1e6}
	f(n_k)=k
	\end{equation}
	for every $k \in \mathbb N \cup\{0\}$. We claim that $f$ has all desirable properties.
		
	$(i)$ \emph{$f$ is unbounded modulus}. It is clear that $f(0)=0$ holds and $f$ is strictly increasing and unbounded. For subadditivity of $f$ it suffices to show that the function $\frac{f(t)}{t}$ is decreasing on $(0, \infty)$. Indeed, if $\frac{f(t)}{t}$ is decreasing, then 
	$$
	f(x+y)=x \frac{f(x+y)}{x+y} + y \frac{f(x+y)}{x+y} \leq x \frac{f(x)}{x} + y \frac{f(y)}{y} = f(x) + f(y).
	$$
	(See, for example, Timan~\cite[3.2.3]{Tim}.) The function $\frac{f(x)}{x}$ is decreasing on $(0,\infty)$ if and only if this function is decreasing on $(n_{k-1}, n_k)$ for every $k \in \mathbb N$. Using~\eqref{L:1e3} we see that the last condition trivially holds on $(n_0, n_1)$, because in this case, the right hand side in~\eqref{L:1e5} is
	$$
	\frac{x-n_0}{n_1-n_0} - (1-1) = \frac{x}{n_1}.
	$$
	Moreover, for $k \geq 2$ the restriction $f|_{(n_{k-1}, n_k)}$ is decreasing if and only if 
	\begin{equation}\label{L:1e7}
	\frac{(k-1)(n_k-n_{k-1}) - n_{k-1}}{n_k-n_{k-1}} \geq 0.
	\end{equation}
	Since, for $k \geq 2$, we have
	$$
	(k-1)(n_k-n_{k-1}) - n_{k-1} \geq n_k-2n_{k-1},
	$$
	the second inequality in~\eqref{L:1e3} implies \eqref{eq2.6}. Thus $f$ is an unbounded modulus.
		
	$(ii)$ \emph{$f$ is concave}. Since $f$ is a piecewise affine function, the one-sided derivatives of $f$ exist at all points $x \in [0, \infty)$. Using~\eqref{L:1e5} and the first inequality in~\eqref{L:1e3} we see that these derivatives are decreasing. Hence $f$ is concave. (For the proof of concavity of functions with decreasing one-sided derivatives see, for example, Artin~\cite[p.~4]{Art}.)
		
	$(iii)$ \emph{$f$ is slowly varying}. It is easy to see that~\eqref{eq2.6} holds for all $a>0$ if it holds for all $a>1$. Since $f$ is increasing, the inequality $a>1$ implies that 
	$$
	\liminf_{x\to \infty} \frac{f(ax)}{f(x)} \geq 1.
	$$
	Thus $f$ is slowly varying if and only if
	\begin{equation}\label{L:1e8}
	\limsup_{x\to \infty} \frac{f(ax)}{f(x)} \leq 1.
	\end{equation}
	Let $a > 1$ and $x >0$. Suppose that
	$$
	x \in [n_{k-1}, n_k] \text{ and } ax \in [n_{k+p}, n_{k+p+1}]
	$$
	for some $p$, $k \in \mathbb N$. It implies that
	\begin{equation}\label{L:1e9}
	a = \frac{ax}{x} \geq \frac{n_{k+p}}{n_{k}}.
	\end{equation}
	Using~\eqref{L:1e7} and~\eqref{L:1e9} we obtain
	\begin{equation}\label{L:1e10}
	(x \in [n_{k-1}, n_k]) \Rightarrow (ax \in [n_{k-1}, n_k] \text{ or } ax \in [n_{k}, n_{k+1}])
	\end{equation}
	for all sufficiently large $x$. Now it follows from~\eqref{L:1e5} and~\eqref{L:1e10} that
	\begin{equation}\label{L:1e11}
	f(ax) \leq f(x)+2.
	\end{equation}
	Since we have $\lim_{x\to \infty} f(x)=\infty$, inequality~\eqref{L:1e11} implies~\eqref{L:1e8}.
		
	$(iv)$ \emph{Equality~\eqref{L:1e1} holds}. We must prove the equality
	\begin{equation}\label{L:1e12}
	\lim_{m\to \infty} \frac{f(\left|K(m)\right|)}{f(m)} = 1.
	\end{equation}
	Let $m \in \mathbb N$ such that $m \geq n_2$. Then there is $k\geq 3$ for which
	\begin{equation}\label{L:1e13}
	n_{k-1} \leq m \leq n_k.
	\end{equation}
	The last double inequality and~\eqref{L:1e6} imply
	\begin{equation}\label{L:1e14}
	k-1 = f(n_{k-1}) \leq f(m) \leq f(n_k) =k.
	\end{equation}
	From~\eqref{L:1e13} it follows that
	\begin{equation}\label{L:1e15}
	\left|K(n_{k-1})\right| \leq \left|K(m)\right| \leq \left|K(n_k)\right|.
	\end{equation}
	Using~\eqref{L:1e4}, \eqref{L:1e15} and the inequality $|K(n_k)| \leq n_k$ we obtain
	$$
	n_{k-2} \leq \left|K(m)\right| \leq n_k,
	$$
	which implies
	\begin{equation}\label{L:1e16}
	k-2 = f(n_{k-2}) \leq \left|K(m)\right| \leq f(n_{k}) = k.
	\end{equation}
	Limit relation~\eqref{L:1e12} follows from~\eqref{L:1e14} and~\eqref{L:1e16}.
\end{proof}

\begin{example}
	The ternary Cantor function $G\colon [0,1] \to [0,1]$ leads to an interesting example of unbounded modulus which is not concave. Indeed, $G$ is subadditive (see, for example, Dobo\v{s}~\cite{Dob} and Timan~\cite[3.2.4]{Tim}) and can be characterized as the unique real-valued, continuous, increasing function $f\colon [0,1]\to \mathbb R$ satisfying the functional equations 
	$$
	f\left(\frac{x}{3}\right) = \frac{1}{2} f(x) \text{ and } f(1-x)=1-f(x)
	$$
	(see Chalice~\cite{Ch} for the proof). Now we define a sequence of functions $G_k$, such that $G_1=G$ and, for every $k \geq 2$, $\operatorname{dom}(G_k) = [0,3^{k-1}]$ and
	$$
	G_k(x) = 2G_{k-1}\left(\frac{x}{3}\right), \quad x \in [0,3^{k-1}].
	$$
	Then the extended Cantor function
	$$
	G_e\colon [0, \infty) \to [0, \infty), \quad G_e(x)=G_k(x), \text{ if } x \in [0,3^{k-1}]
	$$
	is a correctly defined, unbounded modulus which is not concave.
\end{example}

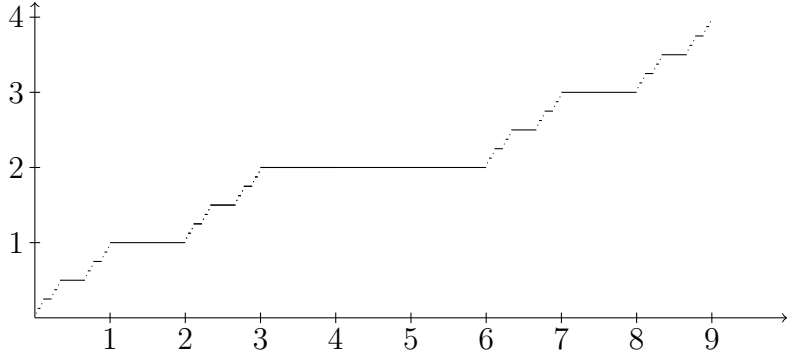
\begin{figure}[h]
	\centering
\begin{tikzpicture}[scale=1]
\draw[->] (0,0)--(10,0);
\draw[->] (0,0)--(0,4.2);
\draw (1/3, 1/2)--(2/3,1/2);
\draw (1/9, 1/4)--(2/9,1/4);
\draw (7/9, 3/4)--(8/9,3/4);
\draw (1/27, 1/8)--(2/27,1/8);
\draw (7/27, 3/8)--(8/27,3/8);
\draw (19/27, 5/8)--(20/27,5/8);
\draw (25/27, 7/8)--(26/27,7/8);
\draw (1/81, 1/16)--(2/81,1/16);
\draw (7/81, 3/16)--(8/81,3/16);
\draw (19/81, 5/16)--(20/81,5/16);
\draw (25/81, 7/16)--(26/81,7/16);
\draw (55/81, 9/16)--(56/81,9/16);
\draw (61/81, 11/16)--(62/81,11/16);
\draw (73/81, 13/16)--(74/81,13/16);
\draw (79/81, 15/16)--(80/81,15/16);
\draw (1,1) --(2,1);
\draw (2+1/3, 1+1/2)--(2+2/3,1+1/2);
\draw (2+1/9, 1+1/4)--(2+2/9,1+1/4);
\draw (2+7/9, 1+3/4)--(2+8/9,1+3/4);
\draw (2+1/27, 1+1/8)--(2+2/27,1+1/8);
\draw (2+7/27, 1+3/8)--(2+8/27,1+3/8);
\draw (2+19/27, 1+5/8)--(2+20/27,1+5/8);
\draw (2+25/27, 1+7/8)--(2+26/27,1+7/8);
\draw (2+1/81, 1+1/16)--(2+2/81,1+1/16);
\draw (2+7/81, 1+3/16)--(2+8/81,1+3/16);
\draw (2+19/81, 1+5/16)--(2+20/81,1+5/16);
\draw (2+25/81, 1+7/16)--(2+26/81,1+7/16);
\draw (2+55/81, 1+9/16)--(2+56/81,1+9/16);
\draw (2+61/81, 1+11/16)--(2+62/81,1+11/16);
\draw (2+73/81, 1+13/16)--(2+74/81,1+13/16);
\draw (2+79/81, 1+15/16)--(2+80/81,1+15/16);
\draw (3,2) --(6,2);
\def\xc{2};
\def\yc{1};
\draw 
({\xc+1},\yc+1) -- ({\xc+2},\yc+1)
({\xc+1/3}, \yc+1/2)--({\xc+2/3},\yc+1/2)
({\xc+1/9}, \yc+1/4)--({\xc+2/9},\yc+1/4)
({\xc+7/9}, \yc+3/4)--({\xc+8/9},\yc+3/4)
({\xc+1/27}, \yc+1/8)--({\xc+2/27},\yc+1/8)
({\xc+7/27}, \yc+3/8)--({\xc+8/27},\yc+3/8)
({\xc+19/27}, \yc+5/8)--({\xc+20/27},\yc+5/8)
({\xc+25/27}, \yc+7/8)--({\xc+26/27},\yc+7/8)
({\xc+1/81}, \yc+1/16)--({\xc+2/81},\yc+1/16)
({\xc+7/81}, \yc+3/16)--({\xc+8/81},\yc+3/16)
({\xc+19/81}, \yc+5/16)--({\xc+20/81},\yc+5/16)
({\xc+25/81}, \yc+7/16)--({\xc+26/81},\yc+7/16)
({\xc+55/81}, \yc+9/16)--({\xc+56/81},\yc+9/16)
({\xc+61/81}, \yc+11/16)--({\xc+62/81},\yc+11/16)
({\xc+73/81}, \yc+13/16)--({\xc+74/81},\yc+13/16)
({\xc+79/81}, \yc+15/16)--({\xc+80/81},\yc+15/16);

\def\xc{6};
\def\yc{2};
\draw 
({\xc+1},\yc+1) -- ({\xc+2},\yc+1)
({\xc+1/3}, \yc+1/2)--({\xc+2/3},\yc+1/2)
({\xc+1/9}, \yc+1/4)--({\xc+2/9},\yc+1/4)
({\xc+7/9}, \yc+3/4)--({\xc+8/9},\yc+3/4)
({\xc+1/27}, \yc+1/8)--({\xc+2/27},\yc+1/8)
({\xc+7/27}, \yc+3/8)--({\xc+8/27},\yc+3/8)
({\xc+19/27}, \yc+5/8)--({\xc+20/27},\yc+5/8)
({\xc+25/27}, \yc+7/8)--({\xc+26/27},\yc+7/8)
({\xc+1/81}, \yc+1/16)--({\xc+2/81},\yc+1/16)
({\xc+7/81}, \yc+3/16)--({\xc+8/81},\yc+3/16)
({\xc+19/81}, \yc+5/16)--({\xc+20/81},\yc+5/16)
({\xc+25/81}, \yc+7/16)--({\xc+26/81},\yc+7/16)
({\xc+55/81}, \yc+9/16)--({\xc+56/81},\yc+9/16)
({\xc+61/81}, \yc+11/16)--({\xc+62/81},\yc+11/16)
({\xc+73/81}, \yc+13/16)--({\xc+74/81},\yc+13/16)
({\xc+79/81}, \yc+15/16)--({\xc+80/81},\yc+15/16);

\def\xc{8};
\def\yc{3};
\draw 
({\xc+1/3}, \yc+1/2)--({\xc+2/3},\yc+1/2)
({\xc+1/9}, \yc+1/4)--({\xc+2/9},\yc+1/4)
({\xc+7/9}, \yc+3/4)--({\xc+8/9},\yc+3/4)
({\xc+1/27}, \yc+1/8)--({\xc+2/27},\yc+1/8)
({\xc+7/27}, \yc+3/8)--({\xc+8/27},\yc+3/8)
({\xc+19/27}, \yc+5/8)--({\xc+20/27},\yc+5/8)
({\xc+25/27}, \yc+7/8)--({\xc+26/27},\yc+7/8)
({\xc+1/81}, \yc+1/16)--({\xc+2/81},\yc+1/16)
({\xc+7/81}, \yc+3/16)--({\xc+8/81},\yc+3/16)
({\xc+19/81}, \yc+5/16)--({\xc+20/81},\yc+5/16)
({\xc+25/81}, \yc+7/16)--({\xc+26/81},\yc+7/16)
({\xc+55/81}, \yc+9/16)--({\xc+56/81},\yc+9/16)
({\xc+61/81}, \yc+11/16)--({\xc+62/81},\yc+11/16)
({\xc+73/81}, \yc+13/16)--({\xc+74/81},\yc+13/16)
({\xc+79/81}, \yc+15/16)--({\xc+80/81},\yc+15/16);

\foreach \x in {1,2,...,9} \draw (\x,-2pt)--(\x,2pt) node [below, yshift=-2pt] {$\x$};
\foreach \x in {1,2,...,4} \draw (-2pt, \x)--(2pt,\x) node [left, xshift=-2pt] {$\x$};
\end{tikzpicture}
\caption{The graph of $G_e$}
\end{figure}

Let us denote by $MUCS$ the set of all unbounded, concave and slowly varying moduli.

\begin{theorem}\label{T:4}
	Let $(X,\rho)$ be a metric space, $(A_k) \subset CL(X)$ and $A \in CL(X)$. Then the following statements are equivalent:
	\begin{enumerate}
		\item [$(i)$] $(A_k)$ is Wijsman convergent to $A$;
		\item [$(ii)$] The equality 
		\begin{equation}\label{T:4e1}
		[WS^f] - \lim A_k = A
		\end{equation}
		holds for every unbounded modulus $f$;
		\item [$(iii)$] Equality~\eqref{T:4e1} holds for every $f \in MUCS$.
	\end{enumerate}
\end{theorem}

\begin{proof}
	$(i) \Rightarrow (ii)$ Let $(i)$ hold. Since $(A_{k})$ is Wijsman convergent to $A$, the set
	$$
	K_{x,\varepsilon}:=\{k \in \mathbb{N}\colon |d(x, A_k) - d(x,A)|\geq \varepsilon\}
	$$
	is finite for all $x \in X$ and $\varepsilon>0$. Let $f\colon [0,\infty) \to [0,\infty)$ be an unbounded modulus. The equality
	\begin{equation*}
	\lim_{n \to \infty}\frac{f(|K_{x,\varepsilon}|)}{f(n)} = 0,  
	\end{equation*}
	holds because $f$ is unbounded and increasing. Thus, $[WS^f]-\lim A_k = A$.
	
	$(ii) \Rightarrow (iii)$ It is trivial.
	
	$(iii) \Rightarrow (i)$ Let $(iii)$ hold. Suppose, $(A_k)$ is not Wijsman convergent to~$A$. Then the set $K_{x,\varepsilon}$ is infinite for some $x \in X$ and $\varepsilon > 0$. Now by Lemma \ref{L:1} there exists $f \in MUCS$ such that $d^f(K_{x,\varepsilon}) = 1$, which contradicts~\eqref{T:4e1}.
\end{proof}

\begin{remark}\label{R:3}
	The sequence $(A_k)$ in Example~\ref{R:03} is $f$-Wijsman statistically convergent with $f(x)=x$ but not Wijsman convergent.
\end{remark}

Theorem~\ref{T:4} us to formulate the following problem.

\begin{problem}\label{P2.9}
	Let $M$ be a set of all unbounded modulus. Describe the sets $S \subseteq M$ for which the conditions:
	\begin{itemize}
		\item $(A_k)$ is Wijsman convergent to $A$
	\end{itemize}
	and
	\begin{itemize}
		\item The equality $[WS^f]-\lim A_k = A$ holds for every $f \in S$
	\end{itemize}
	are equivalent for all metric spaces $(X, \rho)$, $(A_k) \subset CL(X)$ and $A \in CL(X)$.
\end{problem}

The following theorem is similar to Theorem~3.1 from~\cite{ab14}.

\begin{theorem}\label{T:5}
Let $(X,\rho)$ be a metric space, $f\colon [0,\infty)\to [0,\infty)$ be an unbounded modulus, $(A_i) \subset CL(X)$ and $A \in CL(X)$. Then 
$$
[WS^f]-\lim A_i = A
$$
holds if and only if, for each $x \in X$, there exists $K_{x} \subseteq \mathbb{N}$ such that 
$$
d^{f}(K_x)= 0 \quad\text{and}\quad \lim_{k \in \mathbb{N}-K_x} d(x, A_i)= d(x,A).
$$
\end{theorem}
\begin{proof}
	For every $K \subseteq \mathbb N$ and $n \in \mathbb N$ we write $K(n)$ for the set 
	$$
	K \cap \{1, \ldots, n\}.
	$$
	Suppose
	\begin{equation}\label{T:5e1}
	[WS^f]-\lim A_i = A
	\end{equation}
	holds. For every $x \in X$ we must find a set $K_x \subseteq \mathbb N$ such that
	\begin{equation}\label{T:5e2}
	\lim_{i\in \mathbb N - K_x} d(x,A_i)=d(x,A)
	\end{equation}
	and
	\begin{equation}\label{T:5e3}
	\lim_{n \to \infty} \frac{f(|K_x(n)|)}{f(n)} = 0
	\end{equation}
	holds. 
	
	Let $x \in X$. For every $j \in \mathbb N$ define the set $B_j \subseteq \mathbb N$ by the rule:
	\begin{equation}\label{T:5e4}
	(i \in B_j) \Leftrightarrow \left(|d(x, A_i)-d(x,A)|\geq \frac{1}{j}\right).
	\end{equation}
	It is clear that $B_{j_1}\subseteq B_{j_2}$ holds whenever $j_2 \geq j_1$. If all $B_j$ are finite, then~\eqref{T:5e2} and~\eqref{T:5e3} are valid with $K_x = \varnothing$. Suppose $B_j$ are infinite for some $j \in \mathbb N$. If there is $B_{j_1}$ satisfying the condition 
	\begin{itemize}
		\item $B_j - B_{j_1}$ is finite for every $j \in \mathbb N$,
	\end{itemize}
	then~\eqref{T:5e2} and~\eqref{T:5e3} follows from~\eqref{T:5e1} with $K_x = B_{j_1}$. (Note that~\eqref{T:5e3} follows from~\eqref{T:5e1}.)
	
	Let us consider the case when, for every $B_j$, there is $l$ such that $B_{l+j}-B_j$ is infinite. Define a sequence $(j_k)\subseteq \mathbb N$ recursively by the rule:
	\begin{itemize}
		\item if $k=1$, then $j_1$ is the smallest $j$ for which $B_j$ is infinite,
		\item if $k\geq 2$, then $j_k$ is the smallest $j$ with infinite $B_j - B_{j_{k-1}}$.
	\end{itemize}
	
	Write $B_1^*:=B_{j_1}$ and, for $k\geq 2$, $B_k^*:=B_{j_k} - B_{j_{k-1}}$. It follows from~\eqref{T:5e4} that 
	\begin{equation}\label{T:5e5}
	(i \in B_1^*) \Leftrightarrow \left(|d(x, A_i)-d(x,A)|\geq \frac{1}{j_1}\right)
	\end{equation}
	and, for $k \geq 2$,
	\begin{equation}\label{T:5e6}
	(i \in B_k^*) \Leftrightarrow \left(\frac{1}{j_{k}} \leq |d(x, A_i)-d(x,A)|< \frac{1}{j_{k-1}}\right).
	\end{equation}
	It is easily seen that $B_{k_1}^*$ and $B_{k_2}^*$ are disjoint for all distinct $k_1$, $k_2 \in \mathbb N$. 
	
	Let $(n_k) \subseteq\mathbb N$ be a infinite strictly increasing sequence. Write
	\begin{equation}\label{T:5e7}
	B^* :=\bigcup_{k=1}^{\infty} (B_k^* - \{1, \ldots, n_k\}).
	\end{equation}
	
	We claim that~\eqref{T:5e2} holds with $K_x = B^*$. To prove~\eqref{T:5e2} it is suffices to show that the set
	\begin{equation}\label{T:5e8}
	K_{x,\varepsilon}^* := \{i\in (\mathbb N-B^*)\colon |d(x,A_i)-d(x,A)|\geq \varepsilon\}
	\end{equation}
	is finite for every $\varepsilon >0$. If $\varepsilon >0$, then we have either 
	\begin{equation}\label{T:5e9}
	\varepsilon \geq \frac{1}{j_1}
	\end{equation}
	or there is $k\geq 2$ such that 
	\begin{equation}\label{T:5e10}
	\frac{1}{j_{k-1}} > \varepsilon \geq \frac{1}{j_k}.
	\end{equation}
	Let $\varepsilon \geq \frac{1}{j_1}$ and let $i \in K_{x,\varepsilon}^*$. Then $i\in (\mathbb N-B^*)$ and
	\begin{equation}\label{T:5e11}
	|d(x,A_i)-d(x,A)|\geq \frac{1}{j_1}
	\end{equation}
	hold. Since 
	$$
	\mathbb N-B^* = \bigcap_{k=1}^{\infty} (\{1,\ldots, n_k\} \cup (\mathbb N-B_k^*)),
	$$
	the condition $i \in \mathbb N-B^*$ implies
	$$
	i \in \{1,\ldots, n_1\} \text{ or } i \in (\mathbb N-B_1^*).
	$$
	If $i \in (\mathbb N-B_1^*)$, then using~\eqref{T:5e5} we obtain 
	$$
	|d(x,A_i)-d(x,A)|< \frac{1}{j_1},
	$$
	which contradicts~\eqref{T:5e11}. Hence $i \in \{1,\ldots, n_1\}$ holds. Thus if $\varepsilon \geq \frac{1}{j_1}$, then $K_{x,\varepsilon}^*$ is finite with $|K_{x,\varepsilon}^*| \leq n_1$. Similarly if
	$$
	\frac{1}{j_{k-1}} > \varepsilon \geq \frac{1}{j_k} \text{ with } k \geq 2,
	$$
	then, using~\eqref{T:5e6} instead of~\eqref{T:5e5}, we can prove the inequality
	$$
	|K_{x,\varepsilon}^*| \leq n_k.
	$$
	Limit relation~\eqref{T:5e2} follows.
	
	Now we prove that there exists an increasing infinite sequence $(n_k) \subseteq \mathbb N$ such that~\eqref{T:5e3} holds for $K_x=B^*$ with $B^*$ defined by~\eqref{T:5e7}. Equality~\eqref{T:5e1} implies that $d^f(B_j)=0$ holds for every $j \in \mathbb N$. Hence for given $\varepsilon_1>0$ there is $n_1 \in \mathbb N$ such that 
	$$
	\frac{f(\left|B_{j_1}(n)\right|)}{f(n)} \leq \varepsilon_1
	$$
	is valid for every $n \geq n_1$. Let $0< \varepsilon_2 \leq \frac{1}{2}\varepsilon_1$. Using the equality $d^f(B_{j_2})=0$ we can find $n_2 > n_1$ such that
	$$
	\frac{f(\left|B_{j_2}(n)\right|)}{f(n)} \leq \varepsilon_2
	$$
	for all $n \geq n_2$. By induction on $k$ we can find $n_k > n_{k-1}$ which satisfies
	$$
	\frac{f(\left|B_{j_k}(n)\right|)}{f(n)} \leq \frac{1}{2}\varepsilon_{k-1} \leq \left(\frac{1}{2}\right)^{k-1}\varepsilon_{1}
	$$
	for all $n \geq n_k$. It follows~\eqref{T:5e7}, that, for every $k \in\mathbb N$, the inclusion 
	$$
	B^*(n) \subseteq B_{j_k} (n)
	$$
	holds if $n \in [n_{k+1}, n_k)$. Hence we have 
	$$
	\frac{f(\left|B^*(n)\right|)}{f(n)} \leq \left(\frac{1}{2}\right)^{k-1}\varepsilon_{1}
	$$
	if $n \in [n_{k+1}, n_k)$, $k \in \mathbb N$. The equality 
	$$
	\lim_{n\to\infty} \frac{f(\left|B^*(n)\right|)}{f(n)} = 0
	$$
	follows.
	
	Assume now that, for every $x \in X$, there is $K_x \subset \mathbb N$ such that 
	$$
	d^f(K_x)=0 \text{ and } \lim_{i \in \mathbb N-K_x} d(x, A_i) = d(x,A).
	$$
	Let $x \in X$ and $\varepsilon >0$. Then there is $i_0 \in \mathbb N-K_x$ such that 
	$$
	\left|d(x,A_i) - d(x,A)\right| \leq \varepsilon
	$$
	for all $i \in (\mathbb N-K_x)-\{1, \ldots, i_0\}$. Hence 
	$$
	\{i \in \mathbb N \colon \left|d(x,A_i) - d(x,A)\right| > \varepsilon\} \subseteq K_x \cup \{1, \ldots, i_0\}.
	$$
	Equality $d^f(K_x)=0$ implies $d^f(K_x \cup \{1, \ldots, i_0\})=0$. The limit relation 
	$$
	[WS^f]-\lim A_i = A
	$$
	follows.
\end{proof} 

\section{Wijsman statistical convergence and Wijsman Ces\`{a}ro summability}

The following example shows that Wijsman statistical convergence does not imply Wijsman Ces\`{a}ro summability.

\begin{example}\label{E:2}
Let $(X, \rho) = \mathbb{R}$ with the standard metric and let $(A_k)$ be defined as
\[
A_{k}=
\begin{cases}
\{k\}, &\text{if $k$ is a square,} \\
\{0\}, &\text{otherwise}.
\end{cases}
\]
This sequence is Wijsman statistically convergent to the set $\{0\}$ since
\begin{equation*}
\lim_{n \to \infty}\frac{1}{n}|\{k \leq n\colon |d(x,A_k) - d(x, \{0\})| \geq \varepsilon\}| =0
\end{equation*}
holds for all $x \in \mathbb R$ and $\varepsilon >0$. Now, we show that this sequence is not Wijsman Ces\`{a}ro summable. For the sequence $(\sigma_{k}(0))$ of Ces\`{a}ro means of order one of the sequence $(d(0, A_k))$ we have
\[
\sigma_{k}(0) =
\begin{cases}
\frac{(1^2+2^2+\cdots+n^2)}{n^2}, &\text{if } k = n^2,\ \text{for some}\ n \in \mathbb{N} \\
\frac{(1^2+2^2+\cdots+n^2)}{k}, & \text{if } n^2 < k < (n+1)^2,\ \text{for some}\ n \in \mathbb{N}.
\end{cases}
\]
The sequence $(\sigma_{k}(0))$ is not convergent because
$$
\lim_{n\to \infty} \frac{\sum_{1}^{n} k^2}{n^2} = \lim_{n\to \infty} \frac{1}{6}\, \frac{n(n+1)(2n+1)}{n^2} = \infty.
$$
\end{example}

We now give an example of sequence $(A_k) \subset CL(X)$ such that the sequence $(\sigma_k(x))$ of Ces\`{a}ro means of the sequence $(d(x, A_k))$ has a finite limit for every $x \in X$ but $(A_k)$ is not Wijsman Ces\`{a}ro summable to $A$ for any $A \in CL(X)$.

\begin{example}\label{E:4}
Let $(X, \rho) = \mathbb{R}$ with the standard metric and let $(A_k)$ be defined as
\[
A_{k}=
\begin{cases}
\{-1\}, &\text{if k is even}, \\
\{1\}, &\text{if k is odd}.
\end{cases}
\]
Let $x\in \mathbb R$. For the sequence $(\sigma_{k}(x))$ of Ces\`{a}ro means of order one of the sequence $(d(x,A_k))$ we have
\[
\sigma_{k}(x)=
\begin{cases}
|x|, & \text{if $k$ is even and } x \notin [-1,1], \\
1, & \text{if $k$ is even and } x \in [-1,1], \\
\left|x-\frac{1}{k}\right|, & \text{if $k$ is odd and } x \notin [-1,1], \\
1+\frac{x}{k}, & \text{if $k$ is odd and } x\in [-1,1].
\end{cases}
\]
Consequently
\begin{equation}\label{e2.8}
\lim_{k\to\infty} \sigma_{k}(x)=
\begin{cases}
|x|, & \text{if } x \notin [-1,1],\\
1, & \text{if } x \in [-1,1].
\end{cases}
\end{equation}
No we prove that $(A_k)$ is not Wijsman Ces\`{a}ro summable. Indeed, suppose contrary that there is $A \in CL(X)$ with 
\begin{equation}\label{e2.9}
\lim_{k\to\infty} \sigma_{k}(x)= d(x,A)
\end{equation}
for every $x \in \mathbb R$. Since $A$ is non-empty, there is $x_0\in A$. Using~\eqref{e2.8} and \eqref{e2.9} we obtain
$$
0=d(x_0, A) = \lim_{k\to\infty} \sigma_{k}(x_0) \text{ and } \lim_{k\to\infty} \sigma_{k}(x_0) \geq 1.
$$
Thus $0\geq 1$ which is a contradiction.
\end{example}

\begin{remark}\label{R2.11}
	It seems to be intesting to find a criteria guaranteeing the  Wijsman Ces\`{a}ro summability of $(A_k) \subset CL(X)$ to some $A \in CL(X)$ if the sequence $\bigl(\sigma_k(x)\bigr)$ of Ces\`{a}ro means of $\bigl(d(x,A_k)\bigr)$ is Ces\`{a}ro summable for every $x \in X$.
\end{remark}

In the next theorem we show that the Wijsman statistical convergence implies the Wijsman Ces\`{a}ro summability in case of Wijsman bounded sequences.

\begin{theorem}\label{T:8}
Let $(X,\rho)$ be a metric space, let $A \in CL(X)$ and let $(A_k) \subset CL(X)$. If $(A_k)$ is Wijsman bounded and Wijsman statistically convergent to~$A$, then $(A_k)$ is Wijsman Ces\`{a}ro summable to $A$.
\end{theorem}
\begin{proof}
	Let $\varepsilon >0$, $x \in X$, and let $(A_k)$ be Wijsman bounded. For every $n \in \mathbb N$ define the sets $K_{x,\varepsilon}(n)$, $K_{x,\varepsilon}'(n)$ and $M_x$ as
	\begin{align*}
	K_{x,\varepsilon}(n)& := \{k \leq n\colon |d(x,A_k) - d(x,A)| \geq \varepsilon\},\\
	K_{x,\varepsilon}'(n)& := \{1,\ldots,n\} - K_{x,\varepsilon}(n) \text{ and }  M_x := \sup_{k}|d(x,A_k)|.
	\end{align*}
	Suppose $(A_k)$ is Wijsman statistically convergent to $A$. Then the limit relation
	$$
	\lim_{n \to \infty} \frac{|K_{x,\varepsilon}(n)|}{n} = 0,
	$$
	holds. Now we have
	\begin{multline*}
	\left|d(x,A) - \frac{1}{n}\sum_{k=1}^n d(x,A_k)\right| \leq\frac{1}{n}\sum_{k=1}^{n}|(d(x,A_k) - d(x,A))| \\
	=\frac{1}{n}\left(\sum_{k \in K_{x,\varepsilon}'(n)}|d(x,A_k) - d(x,A)| + \sum_{k\in K_{x,\varepsilon}(n)} |d(x,A_k) - d(x,A)|\right)\\
	\leq \frac{(n- \left|K_{x,\varepsilon}(n)\right|)\varepsilon}{n} + \frac{1}{n}\left|K_{x,\varepsilon}(n)\right|M_x \leq \varepsilon + M_x \frac{\left|K_{x,\varepsilon}(n)\right|}{n}.
	\end{multline*}
	It implies the inequality 
	$$
	\limsup_{n\to \infty} \left|d(x,A) - \frac{1}{n}\sum_{k=1}^n d(x, A_k)\right| \leq \varepsilon.
	$$
	Letting $\varepsilon$ to $0$ we obtain
	$$
	\lim_{n \to \infty}\frac{1}{n}\sum_{k=1}^{n}d(x,A_k)= d(x,A).
	$$
	Since $x$ is an arbitrary point of $X$, $(A_k)$ is Wijsman Ces\`{a}ro summable to~$A$.
\end{proof}

\begin{corollary}\label{T:9}
	Let $(X,\rho)$ be a bounded metric space, $A \in CL(X)$, $(A_k) \subset CL(X)$ and let $f\colon [0,\infty)\to [0,\infty)$ be an unbounded modulus. If 
	$$
	[WS^f]-\lim A_k = A,
	$$
	then $(A_k)$ is Wijsman Ces\`{a}ro summable to $A$.
\end{corollary}

It follows from Theorem~\ref{T:2} and Theorem~\ref{T:8} because in each bounded metric space every sequence of non-empty closed sets is Wijsman bounded.

\section{Wijsman strong Ces\`{a}ro summability with respect to a modulus}

The well-known space $w$ of strongly Ces\`{a}ro summable sequences is defined as:
\begin{equation*}
w := \left\{(x_k)\colon \lim_{n \to \infty}\frac{1}{n}\sum_{k=1}^{n}|x_{k} - l| =0, \text{ for some } l \in \mathbb R\right\}.
\end{equation*}
Maddox~\cite{im86} extended the strong Ces\`{a}ro summabllity to that of strong Ces\`{a}ro summabllity with respect to a modulus $f$ and studied the space
\begin{equation*}
w(f) := \left\{(x_k)\colon \lim_{n \to \infty}\frac{1}{n}\sum_{k=1}^{n}f(|x_{k}-l|)=0,\text{ for some
	$l \in \mathbb R$}\right\}.
\end{equation*}

In the year 2012, Nuray and Rhoades \cite{fr12} introduced the notion of Wijsman strong Ces\`{a}ro summability of sequences of sets and discussed its relation with Wijsman statistical convergence.

In this section, we introduce a new concept of Wijsman strong Ces\`{a}ro summability with respect to a modulus $f$. It is shown that, under certain conditions on $f$, Wijsman strong Ces\`{a}ro summability w.r.t. $f$ implies $f$-Wijsman statistical convergence and that the concepts of $f$-Wijsman statistical convergence and of Wijsman strong Ces\`{a}ro summability w.r.t. $f$ are equivalent for Wijsman bounded sequences.

\begin{definition}\label{D:3}
Let $(X,\rho)$ be a metric space and let $f\colon [0,\infty) \to [0,\infty)$ be a modulus. A sequence $ (A_{k}) \subset CL(X)$ is said to be Wijsman strongly Ces\`{a}ro summable to $A \in CL(X)$ with respect to $f$, if the equality
\begin{equation*}
\lim_{n \to \infty}\frac{1}{n}\sum_{k=1}^{n}f\left(|d(x,A_k) - d(x,A)| \right) =0
\end{equation*}
holds for each $x \in X$. 

We write 
$$
[Ww^{f}]-\lim A_k = A
$$
if $(A_k)$ is Wijsman strongly Ces\`{a}ro summable to $A$ w.r.t. $f$.
\end{definition}

\begin{remark}\label{R3.2}
For $f(x) = x$, the concept of Wijsman strong Ces\`{a}ro summability w.r.t. $f$
reduces to that of Wijsman strong Ces\`{a}ro summability.
\end{remark}

\begin{theorem}\label{T:10}
	Let $(X,\rho)$ be a metric space, $(A_k) \subset CL(X)$, $A \in CL(X)$ and let $f\colon [0,\infty)\to [0,\infty)$ be a modulus. If $(A_k)$ is Wijsman strongly Ces\`{a}ro summable to $A$, then 
	\begin{equation}\label{T:10e1}
	[Ww^f]-\lim A_k = A.
	\end{equation}
\end{theorem}
\begin{proof}
Suppose that
\begin{equation}\label{T:10e2}
\lim_{n \to \infty}\frac{1}{n}\sum_{k=1}^{n}|d(x,A_k) - d(x,A)|=0
\end{equation}
holds for each $x \in X$. Let $\varepsilon > 0$ and choose $\delta \in (0,1)$ such that $f(t) < \varepsilon$ for $t \in [0, \delta]$. Consider 
$$
\sum_{k=1}^{n} f(|d(x,A_k) - d(x,A)|) = \sum_{1} + \sum_{2},
$$
where the first summation is over the set $\{k\leq n\colon |d(x,A_k) - d(x,A)| \leq \delta\}$
and the second is over $\{k\leq n\colon |d(x,A_k) - d(x,A)| > \delta\}$. Then $\sum_{1} \leq n \varepsilon$. To estimate $\sum_{2}$ we use the inequality
\begin{equation*}
\bigl|d(x,A_k) - d(x,A)\bigr| < \frac{\bigl|d(x,A_k) - d(x,A)\bigr|}{\delta} \leq 
\left\lceil|d(x,A_k) - d(x,A)|\delta^{-1}\right\rceil,
\end{equation*}
where $\lceil\cdot\rceil$ is the ceiling function. The modulus functions are increasing and subadditive. Hence
\begin{multline*}
f(\left|d(x,A_k) - d(x,A)\right|) \leq f(1)\left\lceil\left|d(x,A_k) - d(x,A)\right| \delta^{-1}\right\rceil \\
\leq 2f(1) \left|d(x,A_k) - d(x,A)\right|\delta^{-1}
\end{multline*}
holds whenever $|d(x,A_k) - d(x,A)| > \delta$. Thus we have
$$
\sum_{2}\leq 2 f(1) \delta^{-1} \sum_{k=1}^{n} \left|d(x,A_k) - d(x,A)\right|,
$$
which together with $\sum_{1} \leq n\varepsilon$ yields
\[
\frac{1}{n}\sum_{k=1}^{n} f\bigl(\left|d(x,A_k) - d(x,A)\right|\bigr) \leq \varepsilon + 2\, f(1)\, \delta^{-1}\, \frac{1}{n} \sum_{k=1}^{n} |d(x,A_k) - d(x,A)|.
\]
Now using~\eqref{T:10e2} we obtain
$$
\limsup_{n\to \infty} \frac{1}{n}\sum_{k=1}^{n} f\bigl(|d(x,A_k) - d(x,A)|\bigr) \leq \varepsilon.
$$
Equality~\eqref{T:10e1} follows by letting $\varepsilon$ to $0$.
\end{proof}

The next example shows that~\eqref{T:10e1} does not imply that $(A_k)$ is Wijsman strongly Ces\`{a}ro summable to $A$.

\begin{example}\label{E:3}
	Let $(X, \rho)=[0,\infty)$ with the standard metric and let $f(x) = \log(1+x)$. Let us consider a sequence $(A_k)$ defined by
	\[
	A_{k}= \begin{cases}
	\{k\}, &\text{if } k \in \{2^r\colon  r \in \mathbb N\},\\
	\{0\}, & \text{otherwise}.
	\end{cases}
	\]
	Then, for every $x \in [0,\infty)$, we have
\begin{equation}\label{eq3.2}
d(x,A_k)= \begin{cases}
|x-k|, & \text{if } k \in \{2^r\colon  r \in \mathbb N\},\\
x, & \text{otherwise}.
\end{cases}
\end{equation}
For any numerical sequence $(x_i) \subset [0,\infty)$, the limit relation 
$$
\lim_{n\to \infty} \frac{1}{n} \sum_{i=1}^n	f(x_i)=0
$$
holds if and only if 
$$
\lim_{r\to \infty} \frac{1}{2^r}\sum_{i=2^r}^{2^{r+1}-1}	f(x_i)=0.
$$
(See Maddox~\cite[p.~523]{im87}). Hence
$$
[Ww^f]-\lim A_k = \{0\}
$$
holds if and only if we have
\begin{equation}\label{eq3.3}
\lim_{r\to \infty} \frac{1}{2^r}\sum_{k=2^r}^{2^{r+1}-1} \log\Bigl(1+\bigl|d(x,A_k)-d(x,\{0\})\bigr|\Bigr)=0
\end{equation}
for every $x\in [0,\infty)$. Using~\eqref{eq3.2} we see that 
$$
\sum_{k=2^r}^{2^{r+1}-1} \log\Bigl(1+\bigl|d(x,A_k)-d(x,\{0\})\bigr|\Bigr) = \log\Bigl(1+\bigl|\left|x-2^r\right|-x\bigr|\Bigr).
$$
For sufficiently large $r$ we have
$$
1+ \bigl|\left|x-2^r\right|-x\bigr| = 2^r-2x+1.
$$
Consequently the left-hand of~\eqref{eq3.3} is equal to 
$$
\lim_{r\to \infty} \frac{1}{2^r}\log(2^r-2x+1).
$$
The last limit is $0$. Thus $(A_k)$ is Wijsman strongly Ces\`{a}ro summable to $\{0\}$ w.r.t $f$. Now, using~\eqref{eq3.2} we obtain
$$
\frac{1}{2^r}\sum_{k=2^r}^{2^{r+1}-1} \bigl|d(x,A_k) - d(x, \{0\})\bigr| = \frac{2^r-2x}{2^r}
$$
for sufficiently large $r$. Thus
$$
\lim_{r\to \infty} \frac{1}{2^r}\sum_{k=2^r}^{2^{r+1}-1} \bigl|d(x,A_k) - d(x, \{0\})\bigr| =1,
$$
which implies that $(A_k)$ is not Wijsman strongly Ces\`{a}ro summable to $\{0\}$.
\end{example}

The following lemma was proved by Maddox in~\cite{im87}.

\begin{Lemma}\label{L:3.5}
	Let $f\colon [0, \infty) \to [0, \infty)$ be a modulus. Then there is a finite $\lim_{t\to \infty} \frac{f(t)}{t}$ and the equality 
	\begin{equation}\label{L:3.5e1}
	\lim_{t\to \infty} \frac{f(t)}{t} = \inf\{t^{-1}f(t)\colon t \in (0, \infty)\}
	\end{equation}
	holds.
\end{Lemma}
\begin{proof}
	Write
	\begin{equation}\label{L:3.5e2}
	\beta := \inf\{t^{-1}f(t)\colon t \in (0, \infty)\}.
	\end{equation}
	It suffices to show that 
	\begin{equation}\label{L:3.5e3}
	\limsup_{t\to \infty} \frac{f(t)}{t} \leq \beta.
	\end{equation}
	Let $\varepsilon>0$ and let $t_0 \in (0, \infty)$ such that 
	$$
	\beta \geq \frac{f(t_0)}{t_0}-\varepsilon.
	$$
	The last inequality is equivalent to 
	\begin{equation}\label{L:3.5e4}
	f(t_0) \leq t_0 (\beta+\varepsilon).
	\end{equation}
	For every $t \in (0,\infty)$ we have 
	\begin{equation}\label{L:3.5e5}
	t = t_0 \left\lfloor\frac{t}{t_0}\right\rfloor + \left(t-t_0 \left\lfloor\frac{t}{t_0} \right\rfloor\right) \leq \left(t_0\left\lfloor \frac{t}{t_0} \right\rfloor+1\right),
	\end{equation}
	where $\lfloor\cdot\rfloor$ is the floor function. Using the increase and subadditivity of $f$ and \eqref{L:3.5e4}--\eqref{L:3.5e5} we obtain 
	$$
	\frac{f(t)}{t} \leq \frac{f(t_0) \left\lfloor\frac{t}{t_0}\right\rfloor+f(1)}{t} \leq \frac{t_0(\beta+\varepsilon)\left\lfloor\frac{t}{t_0}\right\rfloor + f(1)}{t}
	$$
	for all sufficiently large $t$. Hence
	$$
	\limsup_{t\to \infty} \frac{f(t)}{t} \leq (\beta+\varepsilon) \limsup_{t\to \infty} \frac{t_0 \left\lfloor\frac{t}{t_0}\right\rfloor}{t} = \beta+\varepsilon.
	$$
	Inequality~\eqref{L:3.5e3} follows by letting $\varepsilon$ to $0$.
\end{proof}

\begin{theorem}\label{T:11}
	Let $(X,\rho)$ be a metric space, $A \in CL(X)$ and $(A_k) \subset CL(X)$. If $f\colon [0, \infty) \to [0, \infty)$ is a modulus such that
	\begin{equation}\label{T:11e1}
	\beta:=\lim_{t\to \infty} \frac{f(t)}{t} >0 \text{ and } [Ww^f] - \lim A_k =A,
	\end{equation}
	then $(A_k)$ is Wijsman strongly Ces\`{a}ro summable to $A$.
\end{theorem}
\begin{proof}
	Let a modulus $f$ satisfy condition~\eqref{T:11e1}. By Lemma~\ref{L:3.5} we have 
	$$
	\beta = \inf\{t^{-1}f(t)\colon t >0\}.
	$$
	Consequently 
	\begin{equation}\label{T:11e2}
	f(t)\geq \beta t
	\end{equation}
	holds for every $t \geq 0$. It follows from~\eqref{T:11e2} that
	\[
	\frac{1}{n}\sum_{k=1}^{n} |d(x,A_k) - d(x,A)| \leq 
	\beta^{-1} \frac{1}{n}\sum_{k=1}^{n} f(|\,d(x,A_k) - d(x,A)|),
	\]
	holds for every $x \in X$. Using the second term of~\eqref{T:11e1} we see that $(A_k)$ is Wijsman strongly Ces\`{a}ro summable to $A$.
\end{proof}

\begin{theorem}\label{T:12}
Let $(X,\rho)$ be a metric space, $A \in CL(X)$ and $(A_k)\subset CL(X)$. Suppose that $f\colon [0,\infty) \to [0,\infty)$ is an unbounded modulus which satisfies the inequalities
\begin{equation}\label{T:12e1}
\lim_{t \to \infty}\frac{f(t)}{t}>0 \text{ and }f(xy)\geq c\,f(x)\,f(y)
\end{equation}
with some $c \in (0, \infty)$ for all $x$, $y\in [0,\infty)$. Then the following statements hold:
\begin{enumerate}
\item[$(i)$] If $(A_k)$ is Wijsman strongly Ces\`{a}ro summable to $A$ w.r.t. $f$, then $(A_k)$ is $f$-Wijsman statistically convergent to $A$;
\item [$(ii)$] If $(A_k)$ is Wijsman bounded and $f$-Wijsman statistically convergent to~$A$, then $(A_k)$ is Wijsman strongly Ces\`{a}ro summable to $A$ w.r.t. $f$.
\end{enumerate}
\end{theorem}
\begin{proof}
	Let 
	$$
	K_{x,\varepsilon}(n):=\{k \leq n\colon |d(x,A_k)-d(x,A)|\geq \varepsilon\}
	$$
	for all $x \in X$, $\varepsilon \in (0,\infty)$ and $n \in \mathbb N$. 
	
$(i)$ Let $[Ww^f]-\lim A_k=A$. By subadditivity of moduli we have
$$
\sum_{k=1}^{n} f(|d(x,A_k)- d(x,A)|) \geq f\left(\sum_{k=1}^{n} |d(x,A_k)- d(x,A)|\right)
$$
for every $x \in X$. Using the second inequality from~\eqref{T:12e1} we obtain
\begin{equation*}
f\left(\sum_{k \in K_{x,\varepsilon}(n)} |d(x,A_k)- d(x,A)|\right)\geq f\bigl(\left|K_{x,\varepsilon}(n)\right|\varepsilon\bigr) \geq c f\bigl(\left|K_{x,\varepsilon}(n)\right|\bigr)f(\varepsilon).
\end{equation*}
Hence
\begin{equation}\label{T:12e3}
\frac{1}{n}\sum_{k=1}^{n}\,f(|\,d(x,A_k) - d(x,A)|) \geq c\left(\frac{\,f\left(\left|K_{x,\varepsilon}(n)\right|\right)}{f(n)}\right) \left(\frac{f(n)}{n}\right) \,f(\varepsilon).
\end{equation}
This inequality, the first inequality from~\eqref{T:12e1}, $[Ww^f]-\lim A_k=A$ and $\lim_{\varepsilon \to 0} f(\varepsilon)=0$ imply $[WS^f]-\lim A_k=A$.

$(ii)$ Let $(A_k)$ be Wijsman bounded and let $[WS^f]-\lim A_k=A$. Since $(A_k)$ is Wijsman bounded, we have 
\begin{equation}\label{T:12e2}
M_x :=\sup_{k}\left|d(x,A_k)| + d(x,A)\right|<\infty.
\end{equation}
For all $n \in \mathbb{N}$, $x\in X$ and $\varepsilon >0$, we write $K_{x,\varepsilon}'(n):= \{1,\ldots,n\} - K_{x,\varepsilon}(n)$. Now,
\begin{multline*}
\frac{1}{n}\sum_{k=1}^{n}\,f(|\,d(x,A_k) - d(x,A)|)\\
=\frac{1}{n}\sum_{k \in K_{x,\varepsilon}(n)} f(|d(x,A_k) - d(x,A)|) + \frac{1}{n}\sum_{k \in K_{x,\varepsilon}'(n)} f(|d(x,A_k) - d(x,A)|)\\
\leq \frac{\left|K_{x,\varepsilon}(n)\right|}{n} f(M_x)+ \frac{1}{n}n\,f(\varepsilon).
\end{multline*}
Letting $n \to \infty$ we get
\begin{align*}
\frac{1}{n}\sum_{k=1}^{n}\,f(|\,d(x,A_k) - d(x,A)|) \leq f(\varepsilon),
\end{align*}
in view of Theorem \ref{T:2} and \eqref{T:12e2}. Now the equality
$$
[Ww^f]-\lim A_k=A
$$
follows from $\lim_{\varepsilon \to 0} f(\varepsilon)=0$.
\end{proof}

\begin{remark}\label{R:6}
If we take $f(x) = x$ in Theorem \ref{T:12}, we obtain Theorem $6$ of Nuray and Rhoades $\cite{fr12}$.
\end{remark}

It seems to be interesting to find a solution of the following problem.

\begin{problem}\label{P3.9}
	Find characteristic properties of moduli $f$ for which the equalities $[WS^f] - \lim A_k = A$ and $[Ww^f] - \lim A_k = A$ are equivalent for all bounded metric spaces $(X, \rho)$, $(A_k) \subset CL(X)$ and $A \in CL(X)$.
\end{problem}

\medskip
\noindent\textbf{Acknowledgmets}. The research of the third author was supported by grant of the State Fund for Fundamental Research (project F71/20570) and partially supported by grant 0115U000136 of the Ministry Education and Science of Ukraine.

\begin{footnotesize}

\end{footnotesize}

\noindent
Vinod K. Bhardwaj\\
Department of Mathematics, Kurukshetra University,\\
Kurukshetra-$136119$, INDIA\\
email: \texttt{vinodk\_bhj@rediffmail.com}\\[.2cm]
Shweta Dhawan\\
Department of Mathematics, KVA DAV College for Women,\\
Karnal-$132001$, INDIA\\
email: \texttt{shwetadhawan\_dav@rediffmail.com}\\[.2cm]
Oleksiy A. Dovgoshey\\
Function Theory Department,\\
Institute of Applied Mathematics and Mechanics of NASU,\\
Dobrovolskogo str.~$1$, Slovyansk $84100$, UKRAINE\\
email: \texttt{oleksiy.dovgoshey@gmail.com}

\end{document}